\documentclass[12pt]{article}

\usepackage[margin=1in]{geometry}  
\usepackage{graphicx}              
\usepackage{amsmath}               
\usepackage{amsfonts}              
\usepackage{amsthm}                
\usepackage{enumitem}
\usepackage{hyperref}

\newcommand\blfootnote[1]{%
  \begingroup
  \renewcommand\thefootnote{}\footnote{#1}%
  \addtocounter{footnote}{-1}%
  \endgroup
}

\newtheorem{thm}{Theorem}[section]
\newtheorem{lem}[thm]{Lemma}
\newtheorem{prop}[thm]{Proposition}

\newtheorem{claim}[thm]{Claim}

\theoremstyle{definition}
\newtheorem{definition}{Definition}[section]
\theoremstyle{remark}

\DeclareMathOperator{\Det}{Det}
\DeclareMathOperator{\sgn}{sgn}
\DeclareMathOperator{\Cof}{Cof}
\DeclareMathOperator{\size}{size}
\DeclareMathOperator{\col}{col}
\DeclareMathOperator{\row}{row}

\begin{document}

\nocite{*}

\title{Determinants of Block Matrices with Noncommuting Blocks}
\date{}

\author{Nat Sothanaphan}

\maketitle

\begin{abstract}
Let $M$ be an $mn\times mn$ matrix over a commutative ring $R$.
Divide $M$ into $m \times m$ blocks. Assume that the blocks commute pairwise.
Consider the following two procedures:
(1) Evaluate the $n \times n$ determinant formula at these blocks
to obtain an $m \times m$ matrix, and take the determinant again
to obtain an element of $R$;
(2) Take the $mn \times mn$ determinant of $M$.
It is known that the two procedures give the same element of $R$.
We prove that if only certain pairs of blocks of $M$ commute, then the two procedures
still give the same element of $R$, for a suitable definition of noncommutative
determinants. We also derive from our result further collections of commutativity
conditions that imply this equality of determinants, and we prove that our original
condition is optimal under a particular constraint.
\end{abstract}



\blfootnote{\url{http://dx.doi.org/10.1016/j.laa.2016.10.004}}
\blfootnote{\copyright\: 2018. This manuscript version is made available under the CC-BY-NC-ND 4.0 license \protect\url{http://creativecommons.org/licenses/by-nc-nd/4.0/}}

\section{Introduction}

Let $R$ be a commutative ring. For a positive integer $m$, let $M_{m}(R)$ be the ring of
$m \times m$ matrices with entries in $R$. Fix positive integers $m$ and $n$, and let
$M\in M_{mn}(R)$. We can view $M$ as either an $mn\times mn$ matrix over $R$
or an $n \times n$ \emph{block} matrix of $m \times m$ matrices over $R$. Thus
we can identify $M_{mn}(R)$ with $M_{n}(M_{m}(R))$.
Block matrices behave in an analogous way to usual matrices with regard to
addition, subtraction, and multiplication, so that these operations can be carried out
blockwise or elementwise with no difference. With determinants, however, the
analogy breaks down. In general, the determinants of block matrices are not
even \emph{defined}, since $M_{m}(R)$ is usually not commutative. However, if
all blocks of $M$ commute, then the determinant of $M$ over $M_{m}(R)$ not only
is well-defined (the result will be another matrix in $M_{m}(R)$), but also has a close
relationship with the usual determinant of $M$ over $R$, as we will see later
in this section.

Let $\det_{R}M$ denote the determinant of $M$ as a matrix over $R$.
In \cite{bourbaki}, pp. 546-7, Bourbaki
proves the following result (alternative proofs may be found in \cite{kovacs} and
\cite{silvester}):

\begin{thm} \label{thm:bourbaki}
Let $R$ be a commutative ring and let $S$ be a \emph{commutative} subring of $M_m(R)$.
Then for any matrix $M \in M_{n}(S) \subset M_{mn}(R)$, we have
\begin{equation}\label{eq:det}
{\det}_{R}({\det}_{S}M) = {\det}_{R}M.
\end{equation}
\end{thm}

Our main result strengthens Theorem \ref{thm:bourbaki} by relaxing the hypothesis that
\emph{all pairs of blocks commute}. Specifically, our main result is the following:

\begin{thm} \label{thm:f}
Let $R$ be a commutative ring and let $S$ be a \emph{not-necessarily-commutative}
subring of $M_m(R)$. Then for any matrix $M \in M_{n}(S) \subset M_{mn}(R)$
such that
\begin{equation}\label{eq:f}
M_{ij} \text{ commutes with } M_{kl} \text{ whenever } i\neq 1, k \neq 1, \text{and } j \neq l,
\end{equation}
where $M_{ij}$ is the $(i,j)$ entry of $M$ viewed as an $n \times n$ matrix over $S$, we have
\begin{equation}\label{eq:det2}
{\det}_{R}({\Det}_{S}M) = {\det}_{R}M,
\end{equation}
where ${\Det}_{S}M$ is the \emph{noncommutative} determinant of $M$ over $S$, defined
as in Definition \ref{def:det}.
\end{thm}

For example, when $M$ is a $2 \times 2$ block matrix, Theorem \ref{thm:f} states: if $M = 
\begin{pmatrix} 
A & B \\
C & D
\end{pmatrix}$ and $CD=DC$, then ${\det}_{R}M = {\det}_{R}(AD-BC)$.
This $2 \times 2$ result appears also in \cite{silvester}.

Theorem \ref{thm:f} will be proved in Section \ref{sect:main}. First, in Section
\ref{sect:term}, we establish some terminology. In Section \ref{sect:review}
we reprove Theorem \ref{thm:bourbaki}. Then, by modifying Bourbaki's proof, we prove
Theorem \ref{thm:f} in Section \ref{sect:main}. In Section \ref{sect:variant}
we use Theorem \ref{thm:f} to prove two variants of itself, and in Section \ref{sect:classify}
we apply them to classify all commutativity conditions on $2 \times 2$ matrices that suffice
to imply \eqref{eq:det2}. Finally, in Section \ref{sect:optimal}, we prove that Theorem
\ref{thm:f} is optimal under a certain constraint.

\section{Terminology}
\label{sect:term}

\begin{definition} \label{def:det}
Let $R$ be a (not-necessarily-commutative) ring and let $M$ be a matrix in $M_{n}(R)$.
A \emph{noncommutative determinant} of $M$ over $R$, denoted by $\Det_{R}M$,
is defined by
\begin{equation}
{\Det}_{R}M = \sum_{\pi \in S_{n}} \sgn(\pi) M_{1,\pi(1)}M_{2,\pi(2)}\cdots M_{n,\pi(n)},
\end{equation}
where $S_{n}$ is the symmetric group of degree $n$.
\end{definition}

In fact, there are many ways to define a determinant over a noncommutative ring,
since the order of entries matters, but we choose this one for its simplicity.
In \cite{noncomdet}, p. 2, this choice is called the \emph{row-determinant},
since the entries in each product are ordered according to their row numbers,
in contrast to the \emph{column-determinant}, in which the entries are ordered
according to their column numbers. See \cite{noncomdet} for more properties regarding these
determinants. If all entries commute or if $R$ is a \emph{commutative} ring, then
the order does not matter and a noncommutative determinant reduces
to the usual one.

For each positive integer $n$, let $V_n := \{(i,j) : 1 \leq i,j \leq n \}$.

\begin{definition}
Let $R$ be a ring and $M$ a matrix in $M_{n}(R)$. The \emph{commutativity graph}
of $M$ is the graph on $V_n$ such that there is an edge between vertices $(i,j)$ and
$(k,l)$ if and only if entries $M_{ij}$ and $M_{kl}$ commute.
\end{definition}

\begin{definition}
A \emph{commutativity condition} is a graph with vertex set $V_n$ for some
$n \geq 1$. We call $n$ the \emph{size} of $G$, and denote it by $\size(G)$.
For any ring $R$, matrix $M \in M_{n}(R)$ and commutativity condition $G$
of size $n$, we say that $M$ \emph{satisfies} $G$ if the commutativity graph
of $M$ contains $G$ as a subgraph.
\end{definition}

\begin{definition}
A commutativity condition $G$ of size $n$ is a \emph{sufficient commutativity
condition (SCC)} if for every commutative ring $R$, positive integer $m$, ring
$S \subset M_{m}(R)$, and matrix $M \in M_{n}(S)\subset M_{mn}(R)$,
if $M$ satisfies $G$ as a matrix over $S$, then
\begin{equation}\label{eq:det3}
{\det}_{R}({\Det}_{S}M) = {\det}_{R}M.
\end{equation}
\end{definition}

In other words, a sufficient commutativity condition is a commutativity condition that
guarantees that \eqref{eq:det3} holds for every $M$ that satisfies it. In this paper, when
we talk about a block matrix $M \in M_{n}(S)\subset M_{mn}(R)$ satisfying a
commutativity condition $G$, we always assume that $M$ is viewed as a matrix over $S$.
Thus, we can now restate Theorem \ref{thm:bourbaki} in a succinct form as follows:
for any positive integer $n$, the complete graph on $V_n$ is a sufficient
commutativity condition.

\begin{definition}
Let $\mathcal{C}$ be a collection of commutativity conditions.
For any ring $R$ and any matrix $M \in M_{n}(R)$, we say that $M$ \emph{satisfies}
$\mathcal{C}$ if $M$ satisfies at least one member of $\mathcal{C}$.
\end{definition}

More definitions regarding \emph{operations} on commutativity conditions
and collections of commutativity conditions can be found in Section \ref{sect:variant}.

\section{Review of Bourbaki's Proof}
\label{sect:review}

The following proof of Theorem \ref{thm:bourbaki} appears in \cite{bourbaki}, pp. 546-7.
However, since Bourbaki's notation differs greatly from ours and since this proof
will be used directly in the proof of Theorem \ref{thm:f} in Section \ref{sect:main},
we reproduce the proof in full here.

For any commutative ring $A$, any positive integer $k$ and any matrix $X \in M_{k}(A)$, let
$\Cof_A X$ denote the cofactor matrix of $X$ over $A$; then
\begin{equation} \label{eq:cof}
X (\Cof_{A} X)^t = ({\det}_A X)I_k.
\end{equation}

Denote the $(i,j)$ entry of $\Cof_A X$ by $\Cof_A^{ij}X$.

\begin{proof} [Proof of Theorem \ref{thm:bourbaki}]

We proceed by induction on $n$. The case $n=1$ is trivial. Now suppose that $n>1$.
Let $R[z]$ and $S[z]$ be polynomial rings with variable $z$. We identify $S[z]$ with
$M_{m}(R[z])$. Let $N \in M_{n}(S[z])$ be the matrix with entries
\begin{equation} \label{eq:N}
N_{ij}=M_{ij}+\delta_{ij}zI_m,
\end{equation}
where $M$ and $N$ are viewed as matrices over $S$ and $S[z]$, respectively, and
$\delta_{ij}$ is the Kronecker delta. In other words, $N$ is the matrix derived
from $M$ by adding $z$ to each diagonal entry.

Define $U \in M_{n}(S[z])$ by
\begin{equation} \label{eq:U}
U =
\begin{pmatrix} 
\Cof_{S[z]}^{11}N & 0 & \cdots & 0 \\
\Cof_{S[z]}^{12}N & I_m & &\\
\vdots & & \ddots &  \\
\Cof_{S[z]}^{1n}N &  & & I_m
\end{pmatrix}.
\end{equation}
Notice that $U$ is obtained from the identity matrix by replacing its first column
with the first column of $(\Cof_{S[z]} N)^t$. Equation \eqref{eq:cof} explains the 
first column in the product
\begin{equation} \label{eq:NU}
NU =
\begin{pmatrix} 
{\det}_{S[z]}N & N_{12} & \cdots & N_{1n} \\
0 & N_{22} & \cdots & N_{2n}\\
\vdots & \vdots & \ddots & \vdots \\
0 &  N_{n2} & \cdots & N_{nn}
\end{pmatrix},
\end{equation}
which is obtained from the matrix $N$ by replacing its first column with the first column of
$({\det}_{S[z]} N)I_n$. Now, let
\begin{equation} \label{eq:Q}
Q =
\begin{pmatrix} 
N_{22} & \cdots & N_{2n}\\
\vdots & \ddots & \vdots \\
N_{n2} & \cdots & N_{nn}
\end{pmatrix}.
\end{equation}
By taking the determinant in \eqref{eq:NU},
\begin{equation} \label{eq:deteq}
{\det}_{R[z]}N \cdot {\det}_{R[z]}U = {\det}_{R[z]}({\det}_{S[z]}N) \cdot {\det}_{R[z]}Q.
\end{equation}

We need to show that ${\det}_{R[z]}N = {\det}_{R[z]}({\det}_{S[z]}N)$.
Since $Q \in M_{n-1}(S[z])$, by the induction hypothesis,
\begin{align} \label{eq:deteq2}
{\det}_{R[z]}Q &= {\det}_{R[z]}({\det}_{S[z]}Q ) \nonumber  \\ 
&= {\det}_{R[z]}(\Cof_{S[z]}^{11}N) & \text{(by definition of cofactors)} \\
&= {\det}_{R[z]}U, \nonumber
\end{align}
by the determinant formula for a $2 \times 2$ block matrix whose top-right block is zero.
Since ${\det}_{R[z]}Q$ is a monic polynomial in $z$, it is not a
\emph{divisor of zero} in $R[z]$. Cancelling \eqref{eq:deteq2} in \eqref{eq:deteq},
we have ${\det}_{R[z]}N = {\det}_{R[z]}({\det}_{S[z]}N)$. Setting $z=0$, we obtain
${\det}_{R}M = {\det}_R({\det}_{S}M)$ as needed.
\end{proof}

\section{Main Theorem}
\label{sect:main}

We are now ready to prove Theorem \ref{thm:f}. First, we define the collection of
commutativity conditions in accordance with \eqref{eq:f}.

\begin{definition}
For each $n \geq 1$, let $F_n$ be the graph on $V_n$ with the property that there is
an edge between vertices $(i,j)$ and $(k,l)$ if and only if $i \neq 1, k \neq 1$, and $j \neq l$.
Let $\mathcal{F} := \{F_n : n \geq 1\}$, viewed as a collection of commutativity conditions.
\end{definition}

The meaning of $\mathcal{F}$ is the following: a block matrix $M$ satisfies $\mathcal{F}$ if all of
its pairs of blocks that are neither in the first row nor in the same column commute.
In other words, for each $n \geq 1$, $F_n$ is a complete $n$-partite
graph with sets of vertices $S_1, S_2, \ldots,S_n$ such that $S_j = \{(i,j) : 2 \leq i \leq n \}$
for all $j$. 

With this, Theorem \ref{thm:f} now states:
The collection $\mathcal{F}$ is a collection of sufficient commutativity conditions.

The proof of Theorem \ref{thm:f} mimics the proof of Theorem \ref{thm:bourbaki}
in Section \ref{sect:review}. We first start with a lemma.
For any (not-necessarily-commutative) ring $A$ and any matrix
$X \in M_{k}(A)$, let $\Cof_{A}^{ij}X$ be the $(i,j)$ cofactor of $X$ over $A$,
defined in terms of \emph{noncommutative} determinants as
$$\Cof_{A}^{ij}X = (-1)^{i+j}\Det_{A}^{ij}X,$$
where $\Det_{A}^{ij}X$ is the noncommutative determinant of $X$ after row $i$ and
column $j$ have been removed.

\begin{lem} \label{lem:f}
For any ring $A$, any positive integer $k$ and any matrix $X \in M_{k}(A)$
satisfying $\mathcal{F}$, we have
\begin{equation} \label{eq:cof2}
\begin{pmatrix} 
X_{11} & \cdots & X_{1k} \\
X_{21} & \cdots & X_{2k} \\
\vdots & \ddots & \vdots \\
X_{k1} & \cdots & X_{kk}
\end{pmatrix}
\begin{pmatrix} 
\Cof_{A}^{11}X \\
\Cof_{A}^{12}X \\
\vdots  \\
\Cof_{A}^{1k}X
\end{pmatrix}  =
\begin{pmatrix} 
\Det_{A}X \\
0 \\
\vdots \\
0
\end{pmatrix}.
\end{equation}
\end{lem}

\begin{proof}
Since \eqref{eq:cof2} is the first column of \eqref{eq:cof}, it is true in a commutative ring.
For the topmost entry, we need to prove that
$$X_{11}\Cof_{A}^{11}X+ X_{12}\Cof_{A}^{12}X + \cdots + X_{1k}\Cof_{A}^{1k}X = \Det_{A}X.$$
Since the entries in each product are ordered according to their row numbers, and
since the cofactors and determinants involve entries only from row $2$
to row $n$, the ordering on both sides matches. The two sides must be equal.

Let $2 \leq i \leq k$ be an integer. For the $i$th entry from the top, we need to show that
\begin{equation} \label{eq:cof3}
X_{i1}\Cof_{A}^{11}X + X_{i2}\Cof_{A}^{12}X + \cdots + X_{ik}\Cof_{A}^{1k}X = 0.
\end{equation}
Let $1 \leq j \leq k$. Let $P$ be a monomial in the expansion of $X_{ij}\Cof_{A}^{1j}X$. Then,
$P$ contains only entries of $X$ from
row $2$ to row $n$, no two of which are in the same column. By $\mathcal{F}$ all entries in $P$
commute. We can reorder indeterminates in $P$ according to their row numbers, and if their row numbers
are the same, by their column numbers. This, together with the identity \eqref{eq:cof3} in the
commutative case, implies that it holds in this setting too.
\end{proof}

\begin{proof} [Proof of Theorem \ref{thm:f}]
Assuming Lemma \ref{lem:f}, the proof of Theorem \ref{thm:f} follows the same steps as
the proof of Theorem \ref{thm:bourbaki}. We only need to note two things. First, $M$
over $S$ has the same commutativity graphs as $N$ over $S[z]$.
Second, since $N$ satisfies $\mathcal{F}$,
any two entries of $Q$  that are not in the first row nor in the same column commute. Hence
$Q \in M_{n-1}(S[z])$ \emph{also} satisfies $\mathcal{F}$. This allows us to use the
induction hypothesis.
\end{proof}

The commutativity conditions in Theorem \ref{thm:f} may seem unnatural. However,
in Section \ref{sect:optimal} we will see that the collection $\mathcal{F}$ is \emph{minimal},
in the sense that no edge can be removed from a member of $\mathcal{F}$ so that it is still
a sufficient commutativity condition.

Alternative proofs of Theorem \ref{thm:bourbaki} can be found in \cite{kovacs} and
\cite{silvester}. It may be that these too can be adapted to prove Theorem \ref{thm:f}.

\section{Variants of $\mathcal{F}$}
\label{sect:variant}

We now turn to the general theory of sufficient commutativity conditions
(SCCs). We first introduce two lemmas: \emph{Column Permuting Lemma} and
\emph{Transpose Lemma}. Then, we apply the two lemmas to derive two variants
of $\mathcal{F}$: $\mathcal{F}_\text{side}$ and $\mathcal{F}_\text{down}$.

\subsection{The Column Permuting Lemma}

Let $n$ and $m$ be positive integers with $n \geq m$. For any permutation $\pi \in S_m$,
define $\pi' \in S_n$ by $\pi'(i) = \pi(i)$ for all $i \leq m$ and $\pi'(i) = i$ for all $i > m$.
\begin{definition}
Let $n \geq m$. Let $G$ be a commutativity condition of size $n$ and let $\pi \in S_m$ be
a permutation. Then $G_{\col, \pi}$ is the graph on $V_n$ with edges
$\{ (i,\pi'(j)),(k,\pi'(l)) \}$ for all edges $\{ (i,j),(k,l) \}$ of $G$.

Analogously, $G_{\row, \pi}$ is the graph on $V_n$ with edges
$\{ (\pi'(i),j),(\pi'(k),l) \}$ for all edges $\{ (i,j),(k,l) \}$ of $G$.
\end{definition}

\begin{definition}
Let $G$ be a commutativity condition of size $n$. The \emph{transpose}
of $G$, denoted by $G^t$,
is the graph on $V_n$ with edges $\{ (j,i),(l,k) \}$
for all edges $\{ (i,j),(k,l) \}$ of $G$.
\end{definition}


\begin{definition}
Let $\mathcal{C}$ be a collection of commutativity conditions and let $\pi \in S_n$ be a permutation.
Define the collections
\begin{align*}
\mathcal{C}_{\col, \pi} &:= \{G_{\col, \pi} : G \in \mathcal{C}, \size(G) \geq n \} \\
\mathcal{C}_{\row, \pi} &:= \{G_{\row, \pi} : G \in \mathcal{C}, \size(G) \geq n \} \\
\mathcal{C}^t &:= \{G^t : G \in \mathcal{C} \}.
\end{align*}
\end{definition}

\begin{lem} [Column Permuting Lemma]
Let $n \geq m$. Let $G$ be a commutativity condition of size $n$,
and let $\pi \in S_m$ be a permutation. Then
$$G \text{ is an SCC} \iff G_{\col,\pi} \text{ is an SCC}.$$
\end{lem}
In other words, if we permute groups of vertices in an SCC that correspond to columns
in a matrix, it will still be an SCC. (Unfortunately, permuting columns of $\mathcal{F}$
still results in $\mathcal{F}$.) Since rows and columns are not interchangeable in
noncommutative determinants, this lemma does not work for permutations of rows.

\begin{proof}
Let $G$ be an SCC of size $n$. It suffices to prove the theorem
where $\pi$ is an \emph{adjacent transposition}, since adjacent transpositions generate $S_n$.
Let $\pi = (k \text{ } k+1)$. Let $H := G_{\col,\pi}$. We need to prove that $H$ is an SCC.

Let $R$ be a commutative ring, let $m$ be a positive integer, let $S \subset M_{m}(R)$ be a
ring, and let $M \in M_{n}(S) \subset M_{mn}(R)$ be a matrix satisfying $H$.
We need to prove that
\begin{equation} \label{eq:col1}
{\det}_{R}({\Det}_{S}M) = {\det}_{R}M.
\end{equation}
Let $M'$ be the matrix $M$ with columns $k$ and $k+1$ swapped when viewed
as a matrix over $S$. Since $M'$ satisfies $G$, we have
\begin{equation} \label{eq:col2}
{\det}_{R}({\Det}_{S}M') = {\det}_{R}M'.
\end{equation}
We claim that
\begin{equation} \label{eq:colhelp}
{\Det}_{S}M' = -{\Det}_{S}M.
\end{equation}
Since the equation is true in a commutative ring, and indeterminates in each product are ordered
according to their row numbers, the two sides are equal. By \eqref{eq:colhelp},
\begin{equation} \label{eq:col3}
{\det}_{R}({\Det}_{S}M') = (-1)^m{\det}_{R}({\Det}_{S}M).
\end{equation}
On the other hand, if we view $M$ as a matrix over $R$,
it takes $m$ column swappings to transform $M$ into $M'$. Therefore,
\begin{equation} \label{eq:col4}
{\det}_{R}M' = (-1)^m{\det}_{R}M.
\end{equation}
Combining \eqref{eq:col2}, \eqref{eq:col3}, and \eqref{eq:col4}, we obtain
\eqref{eq:col1}, as desired.
\end{proof}

\subsection{The Transpose Lemma}

\begin{definition}
Let $n$ be a positive integer. Let $G$ and $H$ be commutativity conditions of size $n$.
The \emph{edge-union} of $G$ and $H$, denoted by $G \cup H$, is the graph on $V_n$ such that
$E(G \cup H) = E(G) \cup E(H)$, where $E(G)$ denotes the set of edges of $G$.
\end{definition}

\begin{definition}
Let $\mathcal{C}$ and $\mathcal{D}$ be collections of commutativity conditions.
Define the collection
$$\mathcal{C}+\mathcal{D} := \{G \cup H : G \in \mathcal{C}, H \in \mathcal{D}, \size(G)=\size(H) \}.$$
\end{definition}

We now introduce collections of commutativity conditions relevant to the
Transpose Lemma, though they are not by themselves collections of SCCs.

\begin{definition}
Let $c$ be a positive integer. For any $n \geq c$, let $T_{\col c, n}$ be the graph
on $V_n$ with edges $\{ (i,j),(k,l) \}$ for all $i,j,k,l$ such that:
\begin{itemize}
  \item $i \neq k$, $j \neq l$, $j \neq c$ and $l \neq c$; or
  \item ($j=c$ or $l=c$) and $(i-k)(j-l)>0$
\end{itemize}
Let $\mathcal{T}_{\col c} := \{T_{\col c, n} : n \geq c\}$.
\end{definition}
The meaning of $\mathcal{T}_{\col c}$ is the following. A block matrix $M$
satisfies $\mathcal{T}_{\col c}$ if
\begin{itemize}
  \item all pairs of its blocks that are not in column $c$ and are in different rows and
  different columns commute; and
  \item every block in column $c$ commutes with all blocks \emph{strictly above and to the left}
  or \emph{strictly below and to the right} of it.
\end{itemize}

For each positive integer $r$, define $\mathcal{T}_{\row r} := \mathcal{T}_{\col r}^t$.
The meaning of $\mathcal{T}_{\row r}$ is analogous to that of $\mathcal{T}_{\col c}$ above,
except that ``column $c$'' is replaced by ``row $r$.''

Let $\mathbb{N}$ denote the set of positive integers.
Define the collection
$$\mathcal{T} := \bigcup_{c \in \mathbb{N}} \mathcal{T}_{\col c} \cup
\bigcup_{r \in \mathbb{N}} \mathcal{T}_{\row r}.$$

\begin{lem} [Transpose Lemma]
Let $\mathcal{C}$ be a collection of SCCs. Then
$$\mathcal{C}^t  + \mathcal{T}$$
is also a collection of SCCs.
\end{lem}

In other words, if an SCC is transposed, then combined with a member of $\mathcal{T}$,
the result will be an SCC.
The transpose of an SCC by itself is not necessarily an SCC
because rows and columns are not interchangeable.

\begin{proof}
Let $R$ be a commutative ring, let $S \subset M_{m}(R)$ be a
ring, and let $M \in M_{n}(S) \subset M_{mn}(R)$ be a matrix satisfying $\mathcal{C}^t  + \mathcal{T}$.
Assume that $M$ satisfies $\mathcal{C}^t  + \mathcal{T}_{\col c}$ for some $c$. The proof
of the case where $M$ satisfies $\mathcal{C}^t  + \mathcal{T}_{\row r}$ for some $r$ is analogous.
We need to show that
\begin{equation} \label{eq:tran1}
{\det}_{R}({\Det}_{S}M) = {\det}_{R}M.
\end{equation}
Let $M^t$ be the transpose of $M$ over $R$.
Since $M$ satisfies $\mathcal{C}^t$, $M^t$ satisfies $\mathcal{C}$,
because the $(i,j)$ block of $M^t$ is $(M_{ji})^t$ and transposes of commuting matrices commute. Thus
\begin{equation} \label{eq:tran1.5}
{\det}_{R}({\Det}_{S}M^t) = {\det}_{R}M^t = {\det}_{R}M.
\end{equation}
We claim that
\begin{equation} \label{eq:tran2}
({\Det}_{S}M^t)^t = {\Det}_{S}M.
\end{equation}
If equation \eqref{eq:tran2} is true, then
${\det}_{R}({\Det}_{S}M^t) = {\det}_{R}({\Det}_{S}M)$,
which, together with \eqref{eq:tran1.5}, implies \eqref{eq:tran1}.

We now prove \eqref{eq:tran2}.
Since the $(i,j)$ block of $M^t$ is $(M_{ji})^t$, equation \eqref{eq:tran2} can be expanded into
\begin{equation} \label{eq:tran3}
\sum_{\pi \in S_{n}} \sgn(\pi) M_{\pi(n),n}M_{\pi(n-1),n-1}\cdots M_{\pi(1),1} =
\sum_{\pi \in S_{n}} \sgn(\pi) M_{1,\pi(1)}M_{2,\pi(2)}\cdots M_{n,\pi(n)}.
\end{equation}
Equation \eqref{eq:tran3} is true in a commutative ring. We show that
indeterminates in each product can be rearranged so that their orders match.
Let $P$ be a monomial on the left-hand side of the equation. Let $M_{rc}$ be the matrix entry in column $c$
that appears in $P$. Partition the set of indeterminates in $P$ into four quadrants with respect to
$M_{rc}$, and let $A,B,C,D$ be the set of indeterminates in $P$ in the top-left, top-right, bottom-left,
and bottom-right quadrants respectively.
By the first half of the definition of $\mathcal{T}_{\col c}$, any two indeterminates
in $A \cup B \cup C \cup D$ commute.
Therefore,
$$P = \left( \prod_{X \in B}X \prod_{X \in D}X \right) M_{rc}
\left( \prod_{X \in A}X \prod_{X \in C}X \right).$$
By the second half of the definition of $\mathcal{T}_{\col c}$, $M_{rc}$
commutes with all indeterminates in $A$ and $D$. Hence
$$P = \left( \prod_{X \in B}X \prod_{X \in A}X \right) M_{rc}
\left( \prod_{X \in D}X \prod_{X \in C}X \right).$$
Since indeterminates in $A \cup B \cup C \cup D$ commute, indeterminates
in $P$ can be rearranged according to their
row numbers. This is in accordance with the order on the right-hand side.
\end{proof}

\subsection{The $\mathcal{F}_\text{side}$ and $\mathcal{F}_\text{down}$ Collections}

We now apply the Column Permuting Lemma and the Transpose Lemma to derive from $\mathcal{F}$
two more collections of commutativity conditions.
Throughout, let $\mathbb{N}$ denote the set of positive integers.

\begin{definition}
Let $j \geq 1$. Let $(1 \text{ } j) \in S_j$ be the transposition of numbers $1$ and $j$
(interpret $(1 \text{ } 1)$ as the identity permutation). For any $n \geq j$, let
$$F_{n,\text{side},j} := (F_n^t \cup T_{\col 1,n})_{\col, (1 \text{ } j)}.$$
Then let
$$\mathcal{F}_{\text{side},j} := (\mathcal{F}^t + \mathcal{T}_{\col 1})_{\col, (1 \text{ } j)} \
= \{ F_{n,\text{side},j} : n \geq j \}.$$
\end{definition}

A block matrix $M$ satisfies $\mathcal{F}_{\text{side},j}$ if

\begin{itemize}
  \item all pairs of blocks that are not in column $j$ and are in different rows commute; and
  \item every block in column $j$ commutes with all blocks that are not in column $j$ and are
  \emph{strictly below} it.
\end{itemize}

For example, 
$M=
\begin{pmatrix} 
A & B \\
C & D
\end{pmatrix}$
satisfies $\mathcal{F}_{\text{side},1}$ if $BD=DB$ and $AD=DA$,
and $M$ satisfies $\mathcal{F}_{\text{side},2}$ if $AC=CA$ and $BC=CB$.

Let
$$\mathcal{F}_{\text{side}} := \bigcup_{j \in \mathbb{N}} \mathcal{F}_{\text{side},j}.$$

\begin{definition}
Let $i \geq 1$. For any $n \geq i$, let
$$F_{n,\text{down},i} := F_{n,\text{side},i}^t \cup T_{\row i, n}.$$
Then let
$$\mathcal{F}_{\text{down},i} := \mathcal{F}_{\text{side},i}^t + \mathcal{T}_{\row i} \
= \{ F_{n,\text{down},i} : n \geq i \}.$$
\end{definition}

A matrix $M$ satisfies $\mathcal{F}_{\text{down},i}$ if

\begin{itemize}
  \item all pairs of blocks that are not in row $i$ and are in different columns commute; and
  \item every block in row $i$ commutes with all blocks that are not in row $i$ nor its own column and
  are \emph{strictly above} or \emph{strictly to the right} of it.
\end{itemize}

For example, 
$M=
\begin{pmatrix} 
A & B \\
C & D
\end{pmatrix}$
satisfies $\mathcal{F}_{\text{down},2}$ if $AB=BA$, $BC=CB$ and $AD=DA$.

Let
$$\mathcal{F}_{\text{down}} := \bigcup_{i \in \mathbb{N}} \mathcal{F}_{\text{down},i}.$$

\begin{thm}[Sufficiency of $\mathcal{F}_{\text{side}}$ and $\mathcal{F}_{\text{down}}$]
The collections $\mathcal{F}_{\text{side}}$ and $\mathcal{F}_{\text{down}}$ are collections
of sufficient commutativity conditions.
\end{thm}

\begin{proof}
We apply the Column Permuting Lemma and the Transpose Lemma.
\end{proof}

\section{Classification of SCCs of Size 2}
\label{sect:classify}

In this section, we classify all SCCs of size 2. For convenience,
let the notation
$$ A := (1,1), \; B := (1,2), \; C := (2,1), \; D:= (2,2)$$
denote the four members of $V_2$. The corresponding positions of these letters
in a matrix are:
$\begin{pmatrix}
A & B \\
C & D
\end{pmatrix}$.
Let $G_1$, $G_2$, $G_3$, $G_4$ and $G_5$ be graphs on $V_2$
with the following sets of edges:
\begin{align*}
E(G_1) &:= \{ CD \} \\
E(G_2) &:= \{ AD, BD  \} \\
E(G_3) &:= \{ AC, BC  \} \\
E(G_4) &:= \{ AB, AD, BC \} \\
E(G_5) &:= \{ AB, AC, BD \}.
\end{align*}

\begin{thm} [Classification of SCCs of Size 2] \label{thm:classify}
A commutativity condition $G$ of size 2 is an SCC if and only if it
contains at least one of $G_1$, $G_2$, $G_3$, $G_4$ and $G_5$ as a subgraph. 
\end{thm}

Since $G_1 = F_2$, $G_2 = F_{2,\text{side},1}$, $G_3 = F_{2,\text{side},2}$ and
$G_4 = F_{2,\text{down},2}$,
the results of the previous section imply that they are SCCs.
We now prove that $G_5$ is an SCC. We will use the following result from
Silvester (\cite{silvester}, p. 4):
\begin{prop} \label{prop:silvester}
Let $R$ be a commutative ring and let $A,B,C,D$ be matrices in $M_m(R)$.
Let $M=
\begin{pmatrix}
A & B \\
C & D
\end{pmatrix}$. Then,
\begin{enumerate}[label=(\alph*)]
  \item If $AC=CA$, then ${\det}_{R}M = {\det}_{R}(AD-CB)$;
  \item If $BD=DB$, then ${\det}_{R}M = {\det}_{R}(DA-BC)$;
  \item If $AB=BA$, then ${\det}_{R}M = {\det}_{R}(DA-CB)$.
\end{enumerate}
\end{prop}

Proposition \ref{prop:silvester} also implies
that $G_2$, $G_3$ and $G_4$ are SCCs. For example, for $G_2$,
since $BD=DB$, we have ${\det}_{R}M = {\det}_{R}(DA-BC)$, which is equal
to ${\det}_{R}(AD-BC)$ since $AD=DA$.

\begin{lem} \label{lem:e}
The graph $G_5$ is an SCC.
\end{lem}

\begin{proof}
Let $R$ be a commutative ring, let $m$ be a positive integer, let $A,B,C,D$ be matrices in $M_m(R)$,
and let $M=
\begin{pmatrix}
A & B \\
C & D
\end{pmatrix}$ be a block matrix satisfying $G_5$. By definition,
$$AB=BA, AC=CA \text{ and } BD=DB.$$
We need to prove that
\begin{equation} \label{eq:e}
{\det}_{R}M = {\det}_{R}(AD-BC).
\end{equation}
Let $R[z]$ be the polynomial ring with variable $z$, and let $A',B' \in M_m(R[z])$
be the polynomials $A' := A+zI_m$ and $B' := B+zI_m$. Let $M': = 
\begin{pmatrix}
A' & B' \\
C & D
\end{pmatrix}$.
Then $M'$ also satisfies $G_5$. Since $A'B'=B'A'$, Proposition
\ref{prop:silvester}(c) implies
$${\det}_{R[z]}M' = {\det}_{R[z]}(DA'-CB').$$
Therefore,
\begin{align*}
{\det}_{R[z]}(A'B'){\det}_{R[z]}M' &= {\det}_{R[z]}(A'B'){\det}_{R[z]}(DA'-CB') \\
&= {\det}_{R[z]}(A'B'DA'-A'B'CB' )  \\
&= {\det}_{R[z]}(A'DA'B'-B'CA'B') \\
&\phantom{=} \text{(since $B'$ commutes with $A'$ and $D$,
and $A'$ commutes with $B'$ and $C$)}\\
&= {\det}_{R[z]}(A'D-B'C){\det}_{R[z]}(A'B').
\end{align*}
Now, since
${\det}_{R[z]}(A'B') = {\det}_{R[z]}(A'){\det}_{R[z]}(B')$ is a product of
two monic polynomials in $z$, it is not a divisor of zero in $R[z]$.
Cancelling gives ${\det}_{R[z]}M' = {\det}_{R[z]}(A'D-B'C)$. Setting $z=0$ then gives
\eqref{eq:e}.
\end{proof}

\begin{proof} [Proof of Theorem \ref{thm:classify}]
We have shown that $G_1$ -- $G_5$ are SCCs. It remains to show that
all graphs of size 2 that do not contain any of $G_1$ -- $G_5$ as a subgraph
are not SCCs. Let $G$ be such a graph.
Let $H_1$, $H_2$, $H_3$ and $H_4$ be graphs on $V_2$ with the following sets of edges:
\begin{align*}
E(H_1) &:= \{ AD, BC \} \\
E(H_2) &:= \{ AB, AC, AD \} \\
E(H_3) &:= \{ AB, BC, BD \} \\
E(H_4) &:= \{ AC, BD \}.
\end{align*}
We claim that $G$ is a subgraph of at least one of $H_1$ -- $H_4$.
If this is true, it will suffice to show that $H_1$ -- $H_4$ are not SCCs.

We sketch a proof of the claim above.
Since $G$ does not contain $G_1$ as a subgraph, the
edge $CD$ is not in $E(G)$. We divide into three cases
according to whether edges $AD$ and $BC$ are in $E(G)$.
\begin{itemize}
  \item Case 1: Both $AD$ and $BC$ are in $E(G)$. This forces $G = H_1$;
  \item Case 2: Exactly one of $AD$ and $BC$ is in $E(G)$. This forces $G \subset H_2$ or $G \subset H_3$;
  \item Case 3: Neither $AD$ nor $BC$ is in $E(G)$. Depending on which of the three remaining edges
  is missing, this forces $G \subset H_2$, $G \subset H_3$, or $G \subset H_4$.
\end{itemize}
We now show that $H_1$ -- $H_4$ are not SCCs.

Let $R=\mathbb{R}$, $S=M_2(R)$, and $T=M_3(R)$.
Let $A=
\begin{pmatrix}
1 & 2 \\
3 & 4
\end{pmatrix}$ and 
$B =
\begin{pmatrix}
5 & 6 \\
7 & 8
\end{pmatrix}$.
Let
$$C=
\begin{pmatrix}
1 & &\\
& 1 & \\
& & 2
\end{pmatrix},\;
D =
\begin{pmatrix}
1 & 2 \\
3 & 4 \\
& & 5
\end{pmatrix},\;
E =
\begin{pmatrix}
6 & 7 \\
8 & 9 \\
& & 10
\end{pmatrix} \text{and }
F = 
\begin{pmatrix}
1 & 1 \\
& 1 \\
& & 1
\end{pmatrix}.$$
Consider the matrices
$$M_1=
\begin{pmatrix}
A & B \\
B & A
\end{pmatrix}, \;
M_2=
\begin{pmatrix}
A & B \\
A & B
\end{pmatrix} \text{and }
M_3=
\begin{pmatrix}
C & D \\
E & F
\end{pmatrix}.$$
Then $M_1$, $M_2$ and $M_3$ satisfy $H_1$, $H_4$ and $H_2$, respectively.
However, by computation,
\begin{align*}
{\det}_{R}({\Det}_{S}M_1) = -128 &\neq {\det}_{R}M_1 = 0. \\
{\det}_{R}({\Det}_{S}M_2) = 128 &\neq {\det}_{R}M_2 = 0. \\
{\det}_{R}({\Det}_{T}M_3) =1152 &\neq {\det}_{R}M_3 = 1872.
\end{align*}
Thus $H_1$, $H_4$ and $H_2$ are not SCCs.
The case of $H_3$ follows from the case of $H_2$
by the Column Permuting Lemma.
\end{proof}

\section{Optimality of $\mathcal{F}$}
\label{sect:optimal}

Finally, we prove that $\mathcal{F}$ is optimal in a sense to be explained.
For each positive integer $n$, let $\kappa_n$ be the graph on $V_n$
such that there is an edge between vertices $(i,j)$ and $(k,l)$ if and only
if $i \neq 1$ and $k \neq 1$. In other words, $\kappa_n$ is the complete graph
on the vertex set
$\{ (i,j) : 2 \leq i \leq n, 1 \leq j \leq n\}.$

\begin{thm} [Optimality of $\mathcal{F}$] \label{thm:optimal}
Let $n$ be a positive integer and let $G$ be a subgraph of $\kappa_n$.
Then $G$ is an SCC if and only if $G$ contains $F_n$ as a subgraph.
\end{thm}

We divide the proof of Theorem \ref{thm:optimal} into three parts.
We first reduce the theorem to two cases, and then we prove the theorem
in each case separately.

\subsection{Reduction to Two Cases}

Let $G$ be a subgraph of $\kappa_n$ that does not contain $F_n$ as a subgraph.
We need to show that $G$ is not an SCC.
For any edge $AB$ of $F_n$, let $G_{AB}$ be the graph obtained from $\kappa_n$
by removing the edge $AB$.
Then $G$ is a subgraph of $G_{AB}$ for some $AB$.
Therefore, it suffices to prove that $G_{AB}$ is not an SCC for all edges $AB$.

\begin{claim} \hfill \label{claim:reduce}

(a) If $G_{AB}$ is an SCC where $A$ and $B$ are in the same row,
then $G_{(2,1),(2,2)}$ is an SCC.

(b) If $G_{AB}$ is an SCC where $A$ and $B$ are neither in the same row
nor in the same column, then $G_{(2,1),(3,2)}$ is an SCC.
\end{claim}

If Claim \ref{claim:reduce} is true, then it suffices to show that
$G_{(2,1),(2,2)}$ and $G_{(2,1),(3,2)}$ are not SCCs
(which we do in subsequent subsections).
Note that since no edge of $F_n$ connects vertices in the same column,
the two parts of Claim \ref{claim:reduce} exhaust the possibilities of the edge $AB$.

\begin{proof}
We consider part (a). Let $i$ be the row number of $A$ and $B$. Then $i \geq 2$.
Suppose that $G_{AB}$ is an SCC.
By the Column Permuting Lemma, $G_{(i,1),(i,2)}$ is an SCC.
We need to show that $G_{(2,1),(2,2)}$ is an SCC.

Let $R$ be a commutative ring, let $m$ be a positive integer, let $S \subset M_{m}(R)$ be a ring,
and let $M \in M_{n}(S) \subset M_{mn}(R)$ be a matrix satisfying $G_{(2,1),(2,2)}$.
We need to show that
\begin{equation} \label{eq:red1}
{\det}_{R}({\Det}_{S}M) = {\det}_{R}M.
\end{equation}
Let $M'$ be the matrix obtained from $M$ by swapping rows 2 and $i$. Then $M'$ satisfies
$G_{(i,1),(i,2)}$. Since $G_{(i,1),(i,2)}$ is an SCC, we have
\begin{equation} \label{eq:red2}
{\det}_{R}({\Det}_{S}M') = {\det}_{R}M'.
\end{equation}
Similarly to the proof of the Column Permuting Lemma,
if we view $M$ as a matrix over $R$,
it takes $m$ column swappings to transform $M$ into $M'$. Therefore,
\begin{equation} \label{eq:red2.5}
{\det}_{R}M' = (-1)^m{\det}_{R}M.
\end{equation}
We claim that
\begin{equation} \label{eq:red3}
{\Det}_{S}M' = -{\Det}_{S}M.
\end{equation}
If equation \eqref{eq:red3} is true, then
${\det}_{R}({\Det}_{S}M') = (-1)^m{\det}_{R}({\Det}_{S}M)$,
which, together with \eqref{eq:red2} and \eqref{eq:red2.5}, implies \eqref{eq:red1}.

We now prove equation \eqref{eq:red3}.
The two sides are equal in a commutative ring. Moreover, all pairs of matrix entries
that are not in row 1
\emph{commute} except $A$ and $B$. But $A$ and $B$ never appear together in the 
same monomial, since they are in the same row. Thus we can reorder indeterminates
in each product according to their row numbers. This, together with the identity
\eqref{eq:red3} in the
commutative case, implies that it holds in this setting too.

For part (b), the proof is similar to the proof above, so we provide only a sketch.
Let $i$ be the row number of $A$ and let $j$ be the row number of $B$.
Without loss of generality, suppose that $i <j$.
Suppose that $G_{AB}$ is an SCC. By the Column Permuting Lemma, $G_{(i,1),(j,2)}$ is an SCC.
The argument analogous to the above will show that $G_{(2,1),(3,2)}$ is an SCC:
even though $A$ and $B$ do appear together in the same monomial,
they will appear in the same ordering on both sides of equation \eqref{eq:red3}, since we assume $i < j$.
This, together with the fact that all other pairs of indeterminates that are not in row 1 commute, implies
that \eqref{eq:red3} holds.
\end{proof}

\subsection{Proof that $G_{(2,1),(2,2)}$ is not an SCC}
\label{subsect:delta}

Let $R = \mathbb{R}$ and let $S = M_{2}(R)$.
In $S$, let $I$ denote the identity matrix.
Let
$K := \begin{pmatrix} 
1 & \\
& 0
\end{pmatrix}$ and
$ L:= \begin{pmatrix} 
 & 0 \\
1 &
\end{pmatrix}$
in $S$.
Let $a$ be an arbitrary real number.
Let
$A := \begin{pmatrix} 
a &  \\
 & 0
\end{pmatrix}$ and 
$B := \begin{pmatrix} 
 & 1\\
0 & 
\end{pmatrix}$.
Let $M \in M_{n}(S) \subset M_{2n}(R)$ be the following block matrix:
$$M :=\begin{pmatrix}
K & L &  &  & \\
A & B &  &  & \\
 &  & I &  &  \\
 &  & & \ddots &  \\
 &  &  &  & I
\end{pmatrix}.
$$
Since every pair of blocks that are not in row 1 commutes except $A$ and $B$,
the matrix $M$ satisfies $G_{(2,1),(2,2)}$.
We show that the equation
$${\det}_{R}({\Det}_{S}M) = {\det}_{R}M$$
cannot hold for all values of $a$.

By the expansion of the determinant along row 1,
$${\det}_{R} M = \Cof_R^{11} M$$
is \emph{independent} of $a$.
So it suffices to prove that the value of ${\det}_{R}({\Det}_{S}M)$ is \emph{dependent}
on $a$. By the definition of noncommutative determinant,
$${\Det}_{S}M = KB-LA = \begin{pmatrix} & 1 \\ -a &  \end{pmatrix}.$$
Thus ${\det}_{R}({\Det}_{S}M) = a$ is dependent on $a$, as required.

\subsection{Proof that $G_{(2,1),(3,2)}$ is not an SCC}
\label{subsect:epsilon}

Let $R$, $S$, $I$, $K$, $L$, $a$, $A$ and $B$
be defined as in Section \ref{subsect:delta}.
Let $M \in M_{n}(S) \subset M_{2n}(R)$ be the following block matrix:
$$M :=\begin{pmatrix}
K & L &  &  &  \\
A & & I & &  \\
 & B & I & &  \\
 &  &  & \ddots &\\
 & &  &  & I
\end{pmatrix}.
$$
Then $M$ satisfies $G_{(2,1),(3,2)}$.
Similarly, we prove that the equation
\begin{equation} \label{eq:epsilon1}
{\det}_{R}({\Det}_{S}M) = {\det}_{R}M
\end{equation}
cannot hold for all values of $a$.
Similarly,
$${\det}_{R} M = \Cof_R^{11} M$$
is \emph{independent} of $a$.
However,
$${\Det}_{S}M = -KB-LA = \begin{pmatrix} & -1 \\ -a &  \end{pmatrix}.$$
Thus ${\det}_{R}({\Det}_{S}M) = -a$ is dependent on $a$, as required.

\subsection*{Acknowledgements}
I would like to thank Professor Bjorn Poonen for his supervision of this research project,
his advice on the direction of my work, and his comments on the exposition
of this paper.
This research is partially supported by Massachusetts Institute of Technology.

\bibliographystyle{plain}

\bibliography{paper}

\vspace{.2cm}
\noindent
Massachusetts Institute of Technology \\
\href{mailto:natsothanaphan@gmail.com}{\nolinkurl{natsothanaphan@gmail.com}}

\end{document}